\documentclass[reqno]{amsart}
\usepackage[mathscr]{eucal}
\usepackage{MnSymbol}
\usepackage{amsbsy}
\usepackage{amsmath}
\usepackage{amsthm}
\usepackage[colorlinks]{hyperref}
\hypersetup{
 colorlinks=true,
 citecolor=magenta,
 linkcolor=blue,
 urlcolor=cyan}
\usepackage{enumitem}
\usepackage{tikz}
\usetikzlibrary{matrix,arrows,decorations.pathmorphing}

\newtheorem{theorem}{Theorem}
\newtheorem{lemma}[theorem]{Lemma}
\newtheorem{corollary}[theorem]{Corollary}

\theoremstyle{definition}
\newtheorem{example}[theorem]{Example}

\newtheorem{claim}[theorem]{Claim}

\newcommand{\sur}{\mbox{\rm Sur}\hspace{0.02in}}

\newcommand{\PS}{\mbox{\rm PS}\hspace{0.01in}}
\newcommand{\buletita}{{\tiny$\bullet$}}

\begin{document}
\title{Chang Palais-Smale condition and global inversion}
\author{Olivia Gut\'u}
\date{\today}
\email{oliviagutu@mat.uson.mx}
\address{Universidad de Sonora.  Departamento de Matem\'aticas. Blvd. Encinas y Rosales s/n, C. P. 83000. Hermosillo, Sonora, M\'exico.}
\maketitle

\begin{abstract}
Let $f:X\rightarrow Y$ be a  local $C^1$-diffeomorphism between real Banach spaces. We prove that if the locally Lipschitz funcional $x\mapsto \frac{1}{2}| f(x)-y|^2$ satisfies the Chang Palais-Smale condition for all $y\in Y$, then $f$ is a norm-coercive global $C^1$-diffeomorphism. We also give a version of this fact for a weighted Chan Palais-Smale conditon. Finally, we study the relationship of this criterion to some  classical global inversion conditions.
\end{abstract}

\section{Introduction}

 Let $(X,|\cdot|)$ and  $(Y,\pmb{|}\cdot\pmb{|})$ be real Banach spaces. A $C^1$ map $f:X\rightarrow Y$ is said to be a global $C^1$-diffeomorphism provided it is bijective and its inverse is also a $C^1$-map.  Recall, by the Inverse Mapping Theorem,  $f$ is a $C^1$-map such that $df(x)$ is a linear isomorphism for all $x\in X$ if and only if it is  a local $C^1$-diffeomorphism.  An old question motivated by the study of existence and uniqueness of the solutions of a nonlinear equations is: Under which conditions a local $C^1$-diffeomorphism $f$ is a global one?  An old  answer, and one of the must important,  was given by  Banach and  Mazur \cite{banachmazur}, and independently by Caccioppoli \cite{cacciopoli}: $f$ is a global $C^1$-diffeomorphism if and only if $f$ is a proper map. Recall,  $f:X\rightarrow Y$  is said to be {\it proper} if  $f^{-1}(K)$ is compact whenever $K$ is compact. For example, if $X=Y=\mathbb{R}^n$, then $f$ is a proper map if and only if it is  {\it norm-coercive}, namely, $\lim_{|x|\rightarrow+\infty}\pmb{|}f(x)\pmb{|}=+\infty.$  So we get the following theorem, which goes back to Hadamard \cite{hadamard}: A local $C^1$-diffeomorphism $f:\mathbb{R}^n\rightarrow\mathbb{R}^n$ is a global $C^1$-diffeomorphism if and only if it  is norm-coercive. The Hadamard Theorem supports the false  idea that every continuous bijection between vector normed spaces must be norm-coercive. Actually, in the infinite-dimensional setting, every global homeomorphism is a proper map but not necessarily  it is norm-coercive \cite{tseng}.  The main propose of this paper is give a criterion for a local $C^1$-diffeomorphism between real Banach spaces to be a norm-coercive global diffeomorphism.

The proof of the Banach-Mazur theorem and other classical global inversion theorems have exploited the use of covering space techniques via the homotopy-lifting property or some different approach,  but always through some variant of the same argument which includes a sort of monodromy process \cite{plastock,john,ioffe,browder,rheinboldt}.  Nevertherless, since the early nineties, some global inversion conditions that test the effectiveness of the  monodromy argument have  emerged. The crucial hypothesis involves the  Palais-Smale condition of a suitable functional $F$ in terms of  $f$, see e.g.  \cite{rabier2,xaviernollet}. Recall, a $C^1$-functional $F:X\rightarrow\mathbb{R}$ satisfies the Palais-Smale condition if any sequence $\{x_n\}$ in $X$ such that $\{F(x_n)\}$ is bounded and $\|\nabla F(x_n)\|\rightarrow$ 0  contains a convergent subsequence, whose limit is then a critical point of $F$.  Some ideas on the matter have appeared long time before.  For example,   Gordon \cite{gordon} relates the original condition (C) of Palais and Smale \cite{palaismale}  to the Hadamard-Caccippoli Theorem. \footnote{The definition of the Palais-Smale condition given above is a stronger  than the original condition (C) of Palais and Smale but has been widely used in the context of Banach spaces; see Section 3 of  \cite{mawhin2}.}

A  remarkable work in this line correspond to Katriel  who gave a global inverse result for maps $f:X\rightarrow Y$ between certain metric spaces by means of an abstract mountain-pass theorem and an Ekeland variational principle, see Theorem 6.1 in \cite{katriel}. Unlike approach
followed in \cite{rabier2} or \cite{xaviernollet}, Katriel technique works in infinite-dimensional setting as well. Looking closely, from the proof of Theorem 6.1 in \cite{katriel} we can conjecture that: If $(Y,\pmb{|}\cdot\pmb{|})$ is a normed space, $f$ is a local homeomorphism, and the map $x\mapsto \pmb{|}f(x)-y\pmb{|}$ satisfies a  sort of Palais-Smale condition for all $y\in Y$, then $f$ must be bijective. In this spirit, Idczak {\it et al.} \cite{idczak} adapted the ideas of Katriel \cite{katriel}  to get a global inversion theorem for a $C^1$ map: let $f:X\rightarrow Y$ be  a local $C^1$-diffeomorphism such that $X$ is a real Banach space and $Y$ is a real Hilbert space, then $f$ is a global diffeomorphism if, for all $y\in Y$,  the funcional  $F_y(x)=\frac{1}{2}\pmb{|}f(x)-y\pmb{|}^2$ satisfies the Palais-Smale condition ---in such case,  $\frac{1}{2}\pmb{|}\cdot\pmb{|}^2$ is of class $C^1$.

In the first part of this paper, we give a generalization of the Idczak {\it et al.}  result for functions between real Banach spaces  in terms of a weighted version of the Chang Palais-Smale condition; see Theorem \ref{main2}. The Chang Palais-Smale condition was introduced by Chang \cite{chang} in order to extend the variational methods for locally Lipschitz functionals, as it is the mapping $x\mapsto\frac{1}{2}\pmb{|}f(x)-y\pmb{|}^2$, for any real Banach space $(Y,\pmb{|}\cdot\pmb{|})$. The  proposed weighted version  is a particular case of the Zhong's Palais-Smale condition given by Motreanu {\it et al.} \cite{montreanu} for a general class of nonsmooth functionals.  Also, we split the hypothesis of Theorem \ref{main2} as an injectivity and surjectivity criteria in order to clarify the relationship of Theorem \ref{main2} with some classical global inversion theorems; see  Theorem \ref{daggertheorem} and Theorem \ref{doubledaggerheorem}. For example, we prove that  if $\int_0^\infty\inf_{|x|\leq\rho}\frac{1}{\|df(x)^{-1}\|}d\rho=\infty$ (hypothesis of the Hadamard-Levy-Plastock Theorem \cite{hadamard, john, plastock}) then, for every $y\in Y$,  then the map  $F_y(x)=\frac{1}{2}\pmb{|}f(x)-y\pmb{|}^2$ satisfies the weighted Chang Palais-Smale condition for some weight; see Example \ref{ejemplocondicionintegral}. In the second part, we present the whole picture of  Theorem \ref{main2} with respect to other  global inversion theorems, including the Theorem 6.1 of Katriel in \cite{katriel}  and the Banach-Mazur-Caccioppoli Theorem.

\section{Global inversion theorem}

Let $(X,|\cdot|)$ be a real Banach space and let  $X^*$ be its dual space, where $\langle x^*,x\rangle$ denotes the duality for all $x\in X$ and $x^*\in X^*$. Let $F:X\rightarrow \mathbb{R}$ be a locally Lipschitz functional. For each $x\in X$, the generalized directional derivative in direction $v$ is defined as:
$$F^{0}(x; v)=\limsup_{\substack{h\to 0 \\ t\downarrow 0}}\frac{1}{t}\left(F(x+h+tv)-F(x+h)\right).$$
 The map $v\mapsto F^0(x;v)$ is convex and continuous. Let $\partial F(x)$ be the {\it Clarke generalized gradient of $F$ at $x$}. That is, $w^*\in\partial F(x)\subset X^*$ if and only if $\langle w^*,v\rangle\leq F^{0}(x;v)$ for all $v\in X$. Recall, $\partial F(x)$ is a non-empty $w^*$-weakly compact convex set. In order to extend the classical Palais-Smale condition to locally Lipschitz functions, Chang \cite{chang} defined the  lower semi-continuous function: 
$$\lambda_F(x)=\min_{w^*\in\partial F(x)}|w^*|_{X^*}$$

\

\noindent{\bf Chang \PS-condition.} A locally Lipschitz function $F$ satisfies the {\it Chang $\PS$-condition} if any sequence $\{x_n\}$ in $X$ such that $\{F(x_n)\}$ is bounded and $$\lim_{n\rightarrow \infty}\lambda_F(x_n)=0$$ contains a (strongly) convergent subsequence. 

\

Of course, if $F$ is continuously Fr\'echet differentiable map, then the above definition reduces to the usual Palais-Smale condition, since in this case $\partial F(x)=\{dF(x)\}$ and therefore $\lambda_F(x)=\|dF(x)\|$. A typical  extension for  Palais-Smale conditions involves a continuous non-decreasing function $h:[0,+\infty)\rightarrow[0,+\infty)$ such that \begin{equation}\label{integral}\int_0^\infty \frac{1}{1+h(\rho)}d\rho=+\infty.\end{equation} We define:

\

\noindent{\bf Weighted Chang \PS-condition.} A locally Lipschitz function $F$ satisfies the {\it weighted Chang $\PS$-condition} if any sequence $\{x_n\}$ in $X$ such that $\{F(x_n)\}$ is bounded and \begin{equation}\label{limiteacero}\lim_{n\rightarrow \infty}\lambda_F(x_n)(1+h(|x_n|))=0\end{equation} contains a (strongly) convergent subsequence. 

\

If $h=0$, the weighted Chang \PS-condition reduces to  \PS-condition given by Chang. If $h(t)=t$, for all $t\geq 0$ the weighted Chang \PS-condition expresses the extension of the \PS-condition in the sense of Cerami to the locally Lipschitz functionals. In \cite{zhong}, Zhong established a weighted \PS-condition but for G\^ateaux differentiable functionals, in terms of $\|dF(x)\|$ instead of $\lambda_F(x)$. Our weighted Chang \PS-condition is a particular case of a  version of Zhong's Palais-Smale condition given by Motreanu {\it et al.}; see Definition 1.3 of \cite{montreanu}.

As we may expected, if we have a converging sequence $\{x_n\}$ in $X$ that satisfies \eqref{limiteacero}, then the limit is a critical point of $F$. That is:

\begin{lemma}\label{lambdacero}
Let $(X,|\cdot|)$ be a real Banach space and let $F:X\rightarrow\mathbb{R}$ be a locally Lipschitz functional. Let $\{x_n\}\subset X$ be a sequence such that $\lim_{n\rightarrow\infty}x_n=\hat x$ and satisfying \eqref{limiteacero}. Then,  for a continuous non-decreasing function $h:[0,+\infty)\rightarrow[0,+\infty)$ such that \eqref{integral} is fulfilled,  we have $\lambda_F(\hat x)=0$.
\end{lemma}

\begin{proof}
Since $\lambda_F$ is lower semi-continuous,  the map $x\mapsto\lambda_F(x)(1+h(|x|))$ is lower semi-continuous. By  \eqref{limiteacero}, $\lambda_F(\hat x)=0$, since:
$$\lim_{n\rightarrow\infty}\lambda_F(x_n)(1+h(|x_n|))=\liminf_{n\rightarrow\infty}\lambda_F(x_n)(1+h(|x_n|))\geq\lambda_F(\hat x)(1+h(|\hat x|))\geq\lambda_F(\hat x).$$
\end{proof}

Let $f$ be map between real Banach spaces. We are interested in establish a criterion in terms of the weighted Chang \PS-condition in order to ensure the existence and uniqueness of the solutions of a nonlinear equation:
\begin{equation}\label{ecuacionnolineal}f(x)=y.\end{equation} 
Our main result is the following:

%%%MAIN THEOREM%%%
\begin{theorem}\label{main}
Let $(X,|\cdot|)$ and  $(Y,\pmb{|}\cdot\pmb{|})$ be real Banach spaces and let $f:X\rightarrow Y$ be a  local $C^1$-diffeomorphism. Let $y\in Y$. If the locally Lipschitz functional $$F_y(x)=\frac{1}{2}\pmb{|} f(x)-y\pmb{|}^2$$ satisfies the weighted Chang $\PS$-condition for some  continuous nondecreasing function $h:[0,+\infty)\rightarrow[0,+\infty)$ such that \eqref{integral} is fulfilled, then there is a unique solution of the nonlinear equation \eqref{ecuacionnolineal}.
\end{theorem}
%%%%%%%%%%%%%%%

\begin{corollary}[Global Inverse Theorem]\label{main2}
Let $(X,|\cdot|)$ and  $(Y,\pmb{|}\cdot\pmb{|})$ be real Banach spaces and let $f:X\rightarrow Y$ be a local $C^1$-diffeomorphism. Let  $h:[0,+\infty)\rightarrow[0,+\infty)$ be a continuous nondecreasing function such that \eqref{integral} is fulfilled. Suppose that:
\begin{enumerate}
\item[$\star)$] For every $y\in Y$,  the map  $F_y(x)=\frac{1}{2}\pmb{|}f(x)-y\pmb{|}^2$ satisfies the weighted Chang $\PS$-condition for  $h$.
\end{enumerate}
Then $f$ is a norm-coercive diffeomorphism onto $Y$. 
\end{corollary}

\begin{proof}
By Theorem \ref{main},  $f$ is a global diffeomorphism. Furthermore, for every $y\in Y$,  the functional $F_y$ is bounded from below and satisfies the weighted Chang $\PS$-condition, then  $\lim_{|x|\rightarrow+\infty}F_y(x)=+\infty$;  see Corollary 2.4 of \cite{montreanu}. In particular, for $y=0$, we have that $f$ is norm-coercive.
\end{proof}

We have Theorem 3.1 of \cite{idczak} is a particular case of Corollary \ref{main2} of Theorem \ref{main}, with $Y$  a Hilbert space  and $h=0$. The technique of proof of Theorem \ref{main} is basically the same of the proof of Theorem 6.1 of Katriel \cite{katriel}, given in an more abstract setting: a proper mountain-pass lemma  provides the uniqueness of the solution; and the existence of a minimizing sequence converging to a critical point ---given by the Ekeland Variational Principle---  guarantees the existence of the solution. This approach is also the same as in the proof of Theorem 3.1  in \cite{idczak}.
In order to get a simpler and direct proof, we shall use the `made to size'   Lemma \ref{mountainpass} (Theorem 7.2 of \cite{katriel}) instead of the classical mountain-pass theorem of Ambrosetti and Rabinowitz considered in \cite{idczak}.

\begin{lemma}[Schechter-Katriel Mountain-Pass Theorem]\label{mountainpass}
Let $(X,|\cdot|)$ be a Banach space and $F:X\rightarrow\mathbb{R}$ be a locally Lipschitz functional. Suppose that, for some $e\in X$, $e\neq 0$, $r>0$ and $\rho\in\mathbb{R}$:
\begin{enumerate}
\item[\buletita] $F(0)\leq\rho$ and $F(e)\leq\rho$.
\item[\buletita] $|e|\geq r$.
\item[\buletita] $F(x)\geq\rho$ for $|x|=r$.
\end{enumerate}
Then, for a continuous non-decreasing function $h:[0,+\infty)\rightarrow[0,+\infty)$ such that \eqref{integral} is fulfilled, there is a sequence $\{x_n\}\subset X$ such that $F(x_n)\rightarrow c$ for some $c\geq\rho$ and satisfying  \eqref{limiteacero}.
\end{lemma}

 In \cite{chang} Chang proved, by means of a deformation lemma,  that for every bounded from below locally Lipschitz functional  $F$ defined on a reflexive Banach space that satisfies the Chang \PS-condition, the real number $\inf_XF$ is a critical value of $F$. The reflexivity assumption can be  dropped out if we use a standard technique via a convenient variational principle, in this case,   Zhong's variational principle \cite{zhong}; see also Theorem 2.1 in \cite{montreanu}. Actually, we have the following expected fact, which is a version of Corollary 3.3 of Zhong \cite{zhong}  for locally Lipschitz functionals.
  
\begin{lemma}\label{mimimizersequence}
Let $(X,|\cdot|)$ be a Banach space and let  $F:X\rightarrow\mathbb{R}$ be a locally Lipschitz functional and bounded from below. If $h:[0,+\infty)\rightarrow[0,+\infty)$ is a non-decreasing function such that \eqref{integral} is fulfilled, then there is a  sequence $\{x_n\}$ such that $\lim_{n\rightarrow\infty}F(x_n)=\inf_X F$ and satisfying \eqref{limiteacero}. 
\end{lemma}

\begin{proof}
Let $\{\epsilon_n\}$ be a sequence of positive numbers such that $\{\epsilon_n\}\rightarrow 0^+$. Since $F$ is bounded below, by the Ekeland Variational Principle with weight $(1+h(r))^{-1}$,  for every $\epsilon_n>0$  there exists  $x_n$  in $X$  such that,   for all $x\in X$, $F(x_n)<\inf_X F+\epsilon_n $ and $F(x)\geq F(x_n)-\epsilon_n(1+h(|x_n|))^{-1}|x-x_n|.$
Taking $x=x_n+t(u-x_n)$, for an arbitrary $u\in X$ and $t>0$  we have:
$$\frac{F(x_n+t(u-x_n))- F(x_n)}{t}\geq-\frac{\epsilon_n}{1+h(|x_n|)}|u-x_n|.$$
Passing to the upper limit as $t\rightarrow 0^+$, we deduce that:
$$F^0(x_n;u-x_n)\geq\frac{-\epsilon_n}{1+h(|x_n|)}|u-x_n|, \hspace{1cm}\forall u\in X.$$
Let $n$ be a natural number. Then  for $x_n$ fixed, the sets $A_n=\{(v,t):t> F^0(x_n;v)\}$ and $B_n=\{(v,t):t<-\epsilon_n|v|(1+h(|x_n|))^{-1}\}$ are both open and convex; and  $A\cap B=\emptyset$. Therefore, there is a hyperplane passing through zero, given by a linear functional $\xi_n(v,t)=\omega_n(v)+\alpha t$ for some $\alpha\neq 0$,  separating $A$ and $B$. For every $v\in X$, set $\langle w_n^*,v\rangle=-\frac{1}{\alpha}\omega_n(v)$. We have that $\xi_n(v,\langle w_n^*,v\rangle)=0$ for all $v\in X$. Then $\langle w_n^*,v\rangle\leq F^0(x_n;v)$ so $w_n^*\in\partial f(x_n)$. Furthermore, $|\langle w_n^*,v\rangle|(1+h(|x_n|))\leq \epsilon_n|v|$; therefore $\lambda_{F}(x_n)(1+h(|x_n|))\leq|w_n^*|_{X^*}(1+h(|x_n|))\leq \epsilon_n.$  
\end{proof}

Besides the above lemmas, we need to justify first some technical properties in order to give a tidier proof of Theorem \ref{main}. Let $(X,|\cdot|)$ and  $(Y,\pmb{|}\cdot\pmb{|})$ be real Banach spaces and let $f:X\rightarrow Y$  be a  local $C^1$-diffeomorphism. The classical Hadamard-Levy  Theorem,  as well as other classical global inverse results, involves the quantity $\|df(x)^{-1}\|$. For local diffeomorphisms this quantity coincides with the so called {\it Banach constant}:
$$\sur df(x)=\inf_{\pmb{|}v^*\pmb{|}_{Y^*}=1}|df(x)^*v^*|_{X^*}.$$
As it is well known, $df(x_0):X\rightarrow Y$ is  surjective  if and only if   $\sur df(x_0)>0$ if and only if $f$ is {\it open with linear rate around $x_0$}, namely, there exist a neighborhood $V$ of $x_0$ and a constant $\alpha>0$ such that for every $x\in V$ and $r>0$ with $B_r(x)\subset V$ we have $B_{\alpha r}(f(x))\subset f(B_r(x))$. %%BUSCAR REFERENCIA MAS APROPIADA E.G. Metric regularity. theory and applications: a survey.
We shall also considerate in  the proof of Theorem \ref{main}, as well  as in  arguments below, the following facts.
\begin{claim}\label{claim1}
$\lambda_{F_y}(x)\geq\lambda_{\pmb{|}\cdot\pmb{|}}(f(x)-y)\sur df(x)$ for  all $y\in Y$ and   $x\in X$. 
\end{claim}

\begin{proof}
Let $y\in Y$ and  $x\in X$ be fixed.  We have that $\displaystyle\lambda_{F_y}(x)=\min_{w^*\in\partial F_y(x)}|w^*|_{X^*}$ and $\partial F_y(x)=df(x)^*\partial\pmb{|}\cdot\pmb{|}(f(x)-y).$ Let $w^*\in\partial F_y(x)$. Then there is $v^*\in\partial \pmb{|}\cdot\pmb{|}(f(x)-y)$ such that $w^*=df(x)^*v^*$. Note that $\pmb{|}v^*\pmb{|}\geq\min_{v^*\in\partial \pmb{|}\cdot\pmb{|}(f(x)-y)}\pmb{|}v^*\pmb{|}_{Y^*}=\lambda_{\pmb{|}\cdot\pmb{|}}(f(x)-y).$ Suppose that $v^*\neq 0$. Therefore:
\begin{equation}\label{localequation}|w^*|_{X^*}=\frac{|df(x)^*v^*|}{\pmb{|}v^*\pmb{|}_{Y^*}}\pmb{|}v^*\pmb{|}_{Y^*}\geq\sur df(x)\lambda_{\pmb{|}\cdot\pmb{|}}(f(x)-y)\end{equation}
If $v^*=0$ then $\lambda_{\pmb{|}\cdot\pmb{|}}(f(x)-y)=|w^*|_{X^*}=0$ and \eqref{localequation} is satisfied trivially.
Taking the minimum over the set $\partial F_{y}(x)$ we have the desired inequality.
\end{proof}

\begin{claim}\label{claim2}
 $\lambda_{\pmb{|}\cdot\pmb{|}}(f(x)-y)=1$ if $f(x)\neq y$, otherwise  $\lambda_{\pmb{|}\cdot\pmb{|}}(f(x)-y)=0$.
\end{claim}

\begin{proof}
Since $\pmb{|}v^*\pmb{|}_{Y^*}=1$ for $v^*\in\partial \pmb{|}\cdot\pmb{|}(z)$ with $z\neq 0$ we have that if $f(x)\neq y$ then  $\lambda_{\pmb{|}\cdot\pmb{|}}(f(x)-y)=1$. If $f(x)=y$, since the zero functional in $Y^*$ belongs to $\partial\pmb{|}\cdot\pmb{|}(0)$, then $\lambda_{\pmb{|}\cdot\pmb{|}}(f(x)-y)=0$.
\end{proof}

\begin{proof}[Proof of Theorem \ref{main}]
{\it Uniqueness}. Let $y\in Y$ be fixed. Suppose that $f$ is not injective. So, there are two different points $u$ and $e$ in $X$ such that $f(u)=f(e)=y$.   Since  $f$ is open with linear rate around $u$, there exists $\epsilon>0$  such that: \begin{equation}\label{linearopen} B_{\alpha r}(y)\subset f(B_r(u)), \mbox{ for all } 0<r<\epsilon.\end{equation}
Let $r\in (0,\epsilon)$ be small enough such that $f|_{B_r(u)}:B_r(u)\rightarrow f(B_r(u))$ is a diffeomorphism.  Set $\rho=\frac{1}{2}\alpha^2r^2>0$. Suppose that $u=0$. We have that:
\begin{enumerate}
\item[\buletita] $F_y(0)=0\leq\rho$ and $F_y(e)=0\leq\rho$.
\item[\buletita] $|e|\geq r$, since $f|_{B_r(0)}$ is injective.
\item[\buletita] $F_y(x)\geq\rho$ for $|x|=r$, in view of \eqref{linearopen}.
\end{enumerate}
By Lemma \ref{mountainpass}, there is a sequence $\{x_n\}\subset X$ such that $\lim_{n\rightarrow\infty }F_y(x_n)= c$ for some $c\geq\rho$ and satisfying \eqref{limiteacero}. Since $F_y$ satisfies the weighted Chang $\PS$-condition, the sequence $\{x_n\}$ has a convergent subsequence $\{x_{n_k}\}$ with limit $\hat x$. Therefore:
\begin{itemize}
\item[\buletita] $\lambda_{F_y}(\hat x)=0$,  by Lemma \ref{lambdacero}.
\item[\buletita] $f(\hat x)\neq y$, due to $\lim_{k\rightarrow\infty}F_y(x_{n_k})=F_y(\hat x)=c\geq\rho>0$.
\end{itemize}  
By Claim \ref{claim1} and Claim \ref{claim2},  $\sur df(\hat x)=0$. So we get a contradiction.  If $u\neq 0$ then we can consider $G_y(x)=F_y(u-x)$ instead of $F_y(x)$ and carry on an analogous reasoning.

 {\it Existence}. Let $y\in Y$ be fixed.  By Lemma \ref{mimimizersequence}, there is a  sequence $\{x_n\}$ such that $\lim_{n\rightarrow\infty}F_y(x_n)=\inf_X F_y$ and satisfying \eqref{limiteacero}. Since $F_y$ satisfies weighted Chang $\PS$-condition, the sequence $\{x_n\}$ has a convergent subsequence $\{x_{n_k}\}$ with limit $\hat x$.  By Lemma \ref{lambdacero}, we have that $\hat x$ is a critical point of $F_y$. By Claim \ref{claim1} and Claim \ref{claim2}  we have that  $f(\hat x)=y$. 
\end{proof}

From the ideas of the proof of Theorem \ref{main} we can also deduce the following  injectivity criterion:

\begin{theorem}[Injectivity  Criterion]\label{daggertheorem}
Let $(X,|\cdot|)$ and  $(Y,\pmb{|}\cdot\pmb{|})$  be real Banach spaces and let $f:X\rightarrow Y$ be a local $C^1$-diffeomorphism. Let $h:[0,+\infty)\rightarrow[0,+\infty)$ be a continuous nondecreasing function such that \eqref{integral} is fulfilled. If:
\begin{enumerate}
\item[{\rm\dagger)}] For each $y\in Y$, there is no sequence $\{x_n\}$ in $X$ such that $$\lim_{n\rightarrow\infty}F_y(x_n)=c>0 \mbox{ and }\lim_{n\rightarrow \infty}\lambda_{F_y}(x_n)(1+h(x_n))=0.$$
\end{enumerate}
Then $f$ is one-to-one.
\end{theorem}

\begin{proof}
First, suppose that $f$ is not injective. Following step-by-step  the first part of the proof of Theorem \ref{main} we conclude that there is a sequence $\{x_n\}$ in $X$ such that $\lim_{n\rightarrow\infty}F_y(x_n)=c>0 \mbox{ and }\lim_{n\rightarrow \infty}\lambda_{F_y}(x_n)(1+h(x_n))=0$, so we get a contradiction. If $u\neq 0$ consider the sequence $\{x_n'\}=\{u-x_n\}$.
\end{proof}

Furthermore, by the second part of the proof of Theorem \ref{main} we deduce the following result:

\begin{theorem}[Surjectivity  Criterion]\label{doubledaggerheorem}
Let $(X,|\cdot|)$ and  $(Y,\pmb{|}\cdot\pmb{|})$  be real Banach spaces and let $f:X\rightarrow Y$ be a local $C^1$-diffeomorphism. Let $h:[0,+\infty)\rightarrow[0,+\infty)$ be a continuous nondecreasing function such that \eqref{integral} is fulfilled. Suppose that:
\begin{enumerate}
\item[{\rm\ddagger)}] For each $y\in Y$, every sequence $\{x_n\}$ in $X$ such that $$\lim_{n\rightarrow\infty}F_y(x_n)=0 \mbox{ and }\lim_{n\rightarrow \infty}\lambda_{F_y}(x_n)(1+h(x_n))=0$$ has a convergent subsequence. 
\end{enumerate}
Then $f$ is onto.
\end{theorem}

Note that if $f$ is a  local $C^1$-diffeomorphism then $\star)$ is fulfilled if and only if  both conditions, \dagger) and \ddagger), are safistied (with the same $h$). Indeed: suppose that $\star)$ is fulfilled but for a point $y\in Y$ there is a sequence $\{x_n\}$ in $X$ such that $\lim_{n\rightarrow\infty}F_y(x_n)=c$ and $\lim_{n\rightarrow \infty}\lambda_{F_y}(x_n)(1+h(x_n))=0.$ As before, since $F_y$ satisfies the weighted Chang $\PS$-condition, the sequence $\{x_n\}$ has a convergent subsequence $\{x_{n_k}\}$ with limit $\hat x$ such that  $\lambda_{F_y}(\hat x)=0$ and $f(\hat x)\neq y$. So we get the contradiction  $\sur df(\hat x)=0$. Then we have \dagger). Furthermore \ddagger) is also fulfilled by definition of the weighted Chang $\PS$-condition.  Since $F_y$ is always nonnegative, the converse is trivial.

\begin{example}\label{ejemplocondicionintegral}
Let $f$ be as before. Suppose also that $f$ satisfies the Hadamard integral condition, namely (see for instance \cite{hadamard, john, plastock}):
\begin{enumerate}
\item[$\star\star$)] $\lim_{r\rightarrow}\varrho(r)=\infty$ where  $\varrho(r)=\int_0^{r}\inf_{|x|\leq\rho}\frac{1}{\|df(x)^{-1}\|}d\rho.$
\end{enumerate}
Set the nondecreasing map  $h$ given by $\frac{\alpha}{1+h(\rho)}=\inf_{|x|\leq\rho}\frac{1}{\|df(x)\|^{-1}}$ where $\alpha=\|df(0)^{-1}\|^{-1}$. The map $h$ is continuous  since $f$ is $C^1$. Furthermore, condition $\star\star$) implies that  $h$ satisfies \eqref{integral}.  It is easy to see that for all $x\in X$:
$$0<\alpha\leq \sur df(x) (1+h(|x|)).$$
Therefore, by Claim \ref{claim1}, for all $y\in Y$ and $x\in X$:
$$\lambda_{F_y}(x)(1+h(|x|))\geq\lambda_{\pmb{|}\cdot\pmb{|}}(f(x)-y)\alpha\geq 0.$$
Suppose that there is a sequence $\{x_n\}$ in $X$ such that $\lim_{n\rightarrow\infty} F_y(x_n)=c>0$. Then,  without loss of generality, we can assume that $f(x_n)\neq y$ for all natural $n$. Then, by Claim \ref{claim2},  $\lim_{n\rightarrow\infty}\lambda_{F_y}(x_n)(1+h(|x_n|))$ can't be zero. Therefore, \dagger) is satisfied and by Theorem \ref{daggertheorem} the map $f$ is injective. Now, let $\{x_n\}\subset X$  be such that $\lim_{n\rightarrow\infty} F_y(x_n)=0$ and $\lim_{n\rightarrow\infty}\lambda_{F_y}(x_n)(1+h(|x_n|))=0$. Then, by the above inequality, $\lim_{n\rightarrow\infty}\lambda_{\pmb{|}\cdot\pmb{|}}(f(x_n)-y)=0$. Claim \ref{claim2} implies that there exists  $m>0$ such that $f(x_n)=y$ for all $n\geq m$. Since $f$ is injective, this means that $x_n=f^{-1}f(x_n)=f^{-1}(y)$ for all $n\geq m$, so $\{x_n\}$ converges to $f^{-1}(y)$. Therefore \ddagger) is fulfilled. So, we have shown that {\it  if $f$ satisfies the Hadamard integral condition then, for every $y\in Y$,  then the map  $F_y(x)=\frac{1}{2}\pmb{|}f(x)-y\pmb{|}^2$ satisfies the weighted Chang $\PS$-condition} for  $h$ given as above. We can conclude that $f$ is norm-coercive and,  as it is well known, $f$ is a global diffeomorphism.
\end{example}

\section{Relationship to the classical results}
We have shown that Hadamard integral condition satisfies criterion $\star$). We are now interested in locating  condition $\star$) with respect to other  global inversion criteria. As it is known, by the Invariance of Domain Theorem, every mapping  $f:\mathbb{R}^n\rightarrow\mathbb{R}^n$ such that $|f(x)-f(u)|=|x-u|$  for all $u,x\in\mathbb{R}^n$ is a homeomorphism onto $\mathbb{R}^n$. There is an infinite-dimensional version of this fact for Fredholm maps of index zero between real Banach spaces \cite{rheinboldt}. In particular, if $f:X\rightarrow Y$ is a local $C^1$-diffeomorphism and:
\begin{enumerate}
\item[1)] $f$ is a distance preserving isometry
\end{enumerate}
then $f$ is a surjective map (and of course, an injective map). Therefore, $f$ is a global diffeomorphism. Actually, in this case, by the Mazur-Ulam Theorem, $f$ is  an affine transformation.
The distance-preserving hypothesis can be relaxed by the following one  \cite{rheinboldt}:
\begin{enumerate}
\item[2)] $f$ is an expansive map, that is, there is $\alpha>0$ such that for all $x,u\in X$: $$\pmb{|}f(u)-f(x)\pmb{|}\geq \alpha |u-x|.$$
\end{enumerate}
Since if $f$ is local diffeomorphism at $x$,  $\inf_{|v|=1}\pmb{|}df(x)v\pmb{|}=\liminf_{u\rightarrow x}\frac{\pmb{|}f(u)-f(x)\pmb{|}}{|u-x|}$ \cite{john}, both criteria above are a particular case of the Hadamard-Levy hypothesis \cite{levy}:
\begin{enumerate}
\item[3)] There is $\alpha>0$ such that $\pmb{|}df(x)v\pmb{|}\geq \alpha |v|$ for all $x,v\in X$
\end{enumerate}
Which, in turn, is a specific case of  the Hadamard integral condition $\star\star)$, since $\|df(x)^{-1}\|^{-1}=\inf_{|v|=1}\pmb{|}df(x)v\pmb{|}$. Note that by Example \ref{ejemplocondicionintegral}, criterion $\star\star)$ implies that $f$ is norm-coercive and  $\inf_{|x|\leq\rho}\|df(x)^{-1}\|^{-1}<0 \mbox{ for all }\rho>0$. Therefore, if $\star\star)$ is fulfilled then:
\begin{enumerate}
\item[4)] $f$ is norm-coercive and: \begin{equation}\label{pcondition}\sup_{|x|\leq\rho}\|df(x)^{-1}\|<\infty, \mbox{ for all }\rho>0.\end{equation}
\end{enumerate}
The item $4)$ corresponds to another well-known criterion of global inversion \cite{plastock, katriel}. As Katriel noticed in \cite{katriel}, if criterion $4)$  holds then:
\begin{enumerate}
\item[5)] For some  ---hence any--- $y\in Y$ \cite{katriel}: $$\inf\{\mbox{\rm sur}(f,x): \pmb{|}f(x)-y\pmb{|}<\varrho\}>0,\mbox{ for all }\varrho>0.$$ 
\end{enumerate}
Above, $\mbox{sur}(f,x)$ is the  {\it surjection constant} of $f$ at $x$, originally introduced by Ioffe \cite{ioffe}  for non differentiable maps between Banach spaces, namely:
$$\mbox{\rm sur}(f,x)=\liminf_{r\rightarrow 0}\frac{1}{r}\sup\{R\geq 0:B_R(f(x))\subset f(B_r(x))\}.$$
Actually, if $f$ is a local diffeomorphism at $x$ we have $\mbox{sur}(f,x)=\sur df(x).$ This last point corresponds to the hypothesis of the Katriel global inversion theorem. So far we have:\footnote{In \cite{gutu16}, the author presents an analogous relationship between these criteria for mappings between Banach-Finsler manifolds (e.g. Riemannian manifolds).}

\begin{center}
\setlength{\tabcolsep}{0.1em}
\begin{tabular}{cccccccccccc}
1)&$\Rightarrow$ 2)&$\Rightarrow$ & 3) & $\Rightarrow$ & $\star\star)$ & $\Rightarrow$ & $4)$& $\Rightarrow$ & $5)$\\
\end{tabular}
\end{center}
In order to join the above chain  with $\star)$ we shall see that:

\

\begin{proof}[$5)$ implies  $\star)$.]
Let $y\in Y$ be fixed. We will prove that $F_y$ satisfies the Chang \PS-condition. First, we shall see that $\dagger)$ is fulfilled for $h=0$: suppose that there is a sequence $\{x_n\}$ in $X$ such that $\lim_{n\rightarrow \infty}F_y(x_n)=c>0$  and $\lim_{n\rightarrow\infty}\lambda_{F_y}(x_n)=0$.
 Then there exists $\varrho>0$ such that $F_y(x_n)<\varrho$ and, without loss of generality, we can suppose that $f(x_n)\neq y$ for all $n$. Since $f$ satisfies the Katriel condition then  $$\inf\{\sur df(x):F_y(x)<\varrho\}>0.$$ So, $\sur df(x_n)\geq\alpha$ for all $n$ and some $\alpha>0$. Therefore, by Claim \ref{claim1}: $$0<\alpha<\sur df(x_n)\leq\lambda_{F_y}(x_n).$$ Thus $F_y(x_n)\rightarrow 0$, then  we get a contradiction. 

Now, let $\{x_n\}\subset X$  be such that $\lim_{n\rightarrow\infty} F_y(x_n)=0$ and $\lim_{n\rightarrow\infty}\lambda_{F_y}(x_n)=0$. Condition $5)$ implies that $\sur df(x_n)>\alpha>0$ for all $x_n$ such that $f(x_n)\rightarrow y$ for some $\alpha>0$. Then, as in Example \ref{ejemplocondicionintegral}, since $f$ is injective, this means that  \ddagger) is fulfilled. 
\end{proof}

In \cite{rabier} Rabier gives a characterization of  global diffeomorphisms  in terms of sort of a generalized $\mbox{PS}$-condition which is satisfied trivially:\begin{enumerate}
\item[6)] There is no sequence $\{x_n\}$ in $X$ with $f(x_n)\rightarrow y\in Y$ and $\sur df(x_n)\rightarrow 0$.
\end{enumerate}
Rabier proved that a local $C^1$-diffeomorphism is  a global one if and only if 6) is fulfilled. On the other hand, by Banach-Mazur-Caccipioli Theorem,  the local $C^1$-diffeomorphism  $f$ is a diffeomorphism onto $Y$ if and only if it is a proper map (equivalently a closed map).  Summing up, if $f$ is a local $C^1$-diffeomorphism between Banach spaces then:

\

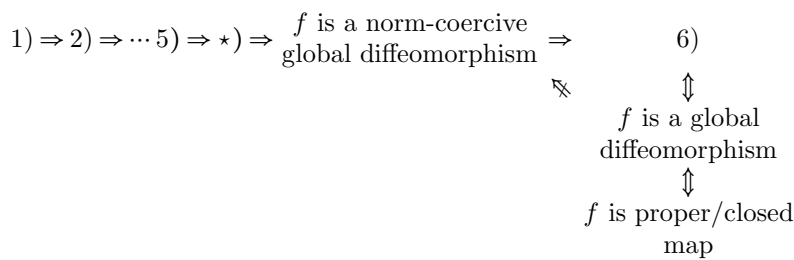
\begin{center}
\setlength{\tabcolsep}{0.1em}
\begin{tabular}{cccccccccccc}
1)&$\Rightarrow$& 2)&$\Rightarrow$&$\cdots$& $5)$&$\Rightarrow$&$\star)$ & $\Rightarrow$& \begin{tabular}{c}$f$ is a norm-coercive\\ global diffeomorphism\end{tabular} &$\Rightarrow$& 6)\\
&&&&&&&&&&$\nNwarrow$&$\Updownarrow$\\
&&&&&&&&&&&\begin{tabular}{c}$f$ is a global \\diffeomorphism\end{tabular}\\
&&&&&&&&&&&$\Updownarrow$\\
&&&&&&&&&&&\begin{tabular}{c}$f$ is proper/closed\\ map\end{tabular}\\
\end{tabular}
\end{center}

\

Now, suppose that  $f:\mathbb{R}^n\rightarrow\mathbb{R}^n$ is a global $C^1$-diffeomorphism, then $f$ is a proper map, and so  is a norm-coercive map. Furthermore, since $x\mapsto \|df(x)^{-1}\|$ is continuous,  \eqref{pcondition} is fulfilled.  Therefore, $f$ satisfies condition 4). So, $4)$, $5)$, $\star$), and $6)$ are all equivalent in the finito-dimensional case.

\bibliographystyle{plainurl}
\bibliography{referenciasDOI}

\end{document}